\documentclass[a4]{amsproc}

\usepackage{graphicx}
\usepackage{amsthm,amsmath,amssymb,mathrsfs}
\usepackage[colorlinks,citecolor=blue,urlcolor=blue]{hyperref}

\numberwithin{equation}{section}
\newtheorem{theorem}{Theorem}[section]
\newtheorem{tm}{Theorem}[section]
\newtheorem{rk}{Remark}[section]

\newtheorem{lm}{Lemma}[section]

\newtheorem{ass}[theorem]{Assumption}

\newcommand{\E}{\mathbb E}
\newcommand{\PP}{\mathbb P}
\newcommand{\N}{\mathbb N}
\newcommand{\R}{\mathbb R}

\newcommand{\C}{\mathcal C}

\newcommand{\HH}{\mathcal H}

\newcommand{\F}{\mathcal F}
\newcommand{\FFF}{\mathscr F}
\newcommand{\<}{\langle}
\renewcommand{\>}{\rangle}

\begin{document}
\title[Optimal H\"older Continuity and Hitting Probabilities]
{Optimal H\"older Continuity and Hitting Probabilities for SPDEs with Rough Fractional Noises} 
\author[]{Jialin Hong}
\thanks{}
\address
{1. LSEC, ICMSEC, 
Academy of Mathematics and Systems Science, Chinese Academy of Sciences, 
Beijing, 100190, China \\ 
2. School of Mathematical Science, University of Chinese Academy of Sciences, Beijing, 100049, China}
\email{hjl@lsec.cc.ac.cn}

\author[]{Zhihui Liu}
\thanks{} \address
{Department of Mathematics, Southern University of Science and Technology, Shenzhen, 518055, China}
\email{liuzh3@sustech.edu.cn}

\author[]{Derui Sheng}
\thanks{} 
\address
{1. LSEC, ICMSEC, 
Academy of Mathematics and Systems Science, Chinese Academy of Sciences, 
Beijing, 100190, China\\ 
2. School of Mathematical Science, University of Chinese Academy of Sciences, Beijing, 100049, China}
\email{sdr@lsec.cc.ac.cn}

\date{\today}

\subjclass[2010]{60H15, 60G22.}

\keywords{
stochastic partial differential equation, 
fractional Brownian sheet, 
Hurst index $H<1/2$,
H\"older exponent,
hitting probability}

\begin{abstract}
We investigate the optimal H\"older continuity and hitting probabilities for systems of stochastic heat equations and stochastic wave equations driven by an additive fractional Brownian sheet with temporal index $1/2$ and spatial index $H\le1/2$. Using stochastic calculus for fractional Brownian motion, we prove that these systems are well-posed and the solutions are H\"older continuous. Furthermore, the optimal H\"older exponents are obtained, which is the first result, as far as we knew, on the optimal H\"older continuity of SHEs and SWEs driven by fractional Brownian sheet that is rough in space. Based on this sharp regularity, we obtain lower and upper bounds of hitting probabilities of the solutions in terms of Bessel--Riesz capacity and Hausdorff measure, respectively.
\end{abstract}

\maketitle

\section{Introduction and Main Results}
\label{sec-1}

Consider
a $d$-dimensional system of stochastic partial differential equations (SPDEs)
\begin{align}\label{spde}
L u_i(t,x)=b_i(u(t,x))+\sum_{j=1}^d \sigma_{i,j}
\frac{\partial^2 W_j}{\partial t\partial x}\quad 
\text{in}\quad (0,T]\times \R,
\end{align}
for $i\in\N_d:=\{1,\ldots,d\}$,
where $T>0, u(t,x)=(u_1(t,x),\ldots,u_d(t,x))$, $\sigma=(\sigma_{i,j})_{i,j\in \N_d}$ is an $\R^{d\times d}$-valued constant matrix, $L=\partial_t-\partial_{xx}$ corresponds to the stochastic heat equation (SHE), or $L=\partial_{tt}-\partial_{xx}$ corresponds to
 the stochastic wave equation (SWE), $b=(b_i)_{i\in \N_d}$ is an $\R^d$-valued function.  In the SHE case we impose $u(0,x)=u_0(x)$, $x \in \R$, while in the SWE case we further impose $u_t(0,x)=v_0(x)$, $x \in \R$. 
 The noise process $W=(W_1,\ldots,W_d)$ is a  centered Gaussian process whose
covariance functional is given by
\begin{align}\label{rkhs-w}
\E[W_i(\varphi)W_j(\psi)]
=\delta_{ij}\int_0^T \int_{\R} \F \varphi(t,\cdot)(\xi) \overline{\F \psi(t,\cdot)(\xi)} \mu(d\xi) dt
\end{align}
for any $\varphi,\psi\in \C_0^\infty([0,T]\times \R)$.
The objective of this paper is to study the optimal H\"older continuity and hitting probability of the solution $u=\{u(t,x),\,t\in[0,T],\,x\in\R\}$ of Eq. \eqref{spde} with the spectral measure $\mu$  given by 
\begin{align}\label{mu}
\mu(d\xi)=c_H|\xi|^{1-2H}d\xi,\quad c_H=\frac{\Gamma(2H+1)\sin(\pi H)}{2\pi},
\end{align}  
where $\Gamma$ is the Gamma function and $H\le 1/2$.

The sample-path H\"older continuity 
 for the solutions of 
SHE and SWE driven by temporally white and spatially homogeneous colored noise has been well studied when $\tilde\Gamma=\F\mu$  is a non-negative tempered measure. 
 For instance,
the authors in \cite{SS02} study the H\"older continuity properties of SHE over $\R^m$, $m\ge 1$ assuming that in \eqref{rkhs-w}, 
$\tilde\Gamma=\F\mu$  is a non-negative tempered measure and 
the spectral measure $\mu$ of the noise satisfies
\begin{align}\label{spe-mea}
\int_{\R^m} \frac{\mu(dx)}{(1+|x|^2)^\eta}&<\infty\quad \text{for some}\quad \eta\in (0,1).
\end{align}
Based on the fractional Sobolev embedding theorem combining the Fourier transformation technique, \cite{DF98, DS05, DS09, HHN14} give the H\"older exponent of SWE over $\R^m$ with $m=1,2,3$ driven by similar noise whose spectral measure is given by a Riesz kernel.
See also \cite{Wal86} for SHE and SWE over bounded intervals driven by space-time white noise with homogeneous Dirichlet or Neumann boundary condition.
We also remark that \cite{HNS11} and, respectively, \cite{HLN12}, applying Feymann-Kac formula, derive the H\"older regularity of SHE driven by  fractional Brownian sheet (FBS) with each component index $H\in (1/2,1)$ and by temporal fractional Brownian motion with $H<1/2$ and spatial homogeneous smooth noise. 
{\color{black}It is known that for the spectral measure $\mu$ defined by \eqref{mu} with $H<1/2$, its Fourier transform
 $\tilde\Gamma=\mathcal F\mu$  given by 
$$\tilde\Gamma(\varphi)=H(2H-1)\int_{\R}(\varphi(x)-\varphi(0))|x|^{2H-2}d x,\quad\varphi\in\C_0^\infty(\R),$$
 is a genuine distribution and is not locally integrable (see e.g. \cite{BJQ15}). 
 }

Due to the complex spatial structure of the FBS 
$W$ determined by \eqref{mu}, the well-posedness of Eq. \eqref{spde} with general Lipschitz continuous diffusion coefficient $\sigma(u)$ is an open problem.
Recently, \cite{BJQ15, BJQ16} establish the existence of a unique mild solution and its H\"older-type estimate of Eq. \eqref{spde} with vanishing drift and affine diffusion, i.e., $b= 0$ and $\sigma(u)=a_1u+a_2$ with $a_1,a_2\in \R$.
For special nonlinear diffusion $\sigma(u)$ which is differentiable with a Lipschitz derivative and satisfies $\sigma(0) = 0$, \cite{HHLNT17(AOP)} obtains a similar well-posed result.
In these multiplicative cases, the Hurst index is restricted as $H\in (1/4,1/2)$ because of technical requirements. 
On the other hand, we recall that \cite{CHL17(SINUM)} investigates the Sobolev regularity of the solution to Eq. \eqref{spde} and Wong--Zakai approximations for the proposed noise to numerically solve Eq. \eqref{spde} with various boundary conditions. 
In this paper,
we need the following assumptions on the drift coefficient $b$ and the initial data $u_0$ and $v_0$.
\begin{ass}\label{ass-b}
$b$ is Lipschitz continuous, i.e., 
\begin{align*}
L_b:=\sup_{u_1\neq u_2} \frac{|b(u_1)-b(u_2)|}{|u_1-u_2|}<\infty.
\end{align*}
\end{ass}

\begin{ass}\label{ass-u0v0}
$u_0$ and $v_0$ are stochastically $\alpha$-H\"{o}lder continuous with $\alpha\in (0,1]$, i.e., for all $p\ge 1$, there exists $L_0=L_0(p)\in (0,\infty)$ such that 
\begin{align*}
\|u_0(x)-u_0(y)\|_{\mathbb L^p}+\|v_0(x)-v_0(y)\|_{\mathbb L^p}
\le L_0 |x-y|^\alpha,\quad x,y\in \R.
\end{align*}
\end{ass}
Here and after, for $v\in L^p(\Omega;\R^d)$, we denote $\|v\|_{\mathbb L^p}:=(\E[\|v\|^p])^\frac 1p,$ for $p\ge 1$. In order to characterize the regularity of the solution of Eq. \eqref{spde}, we introduce 
the following parabolic and hyperbolic metrics: 
\begin{align*}
\triangle((t,x);(s,y)):
=\begin{cases}
|t-s|^{\frac 12}+|x-y|,&\quad \text{for SHE};\\
|t-s|+|x-y|,&\quad \text{for SWE},
\end{cases}
\end{align*}
for $(t,x),(s,y)\in [0,T]\times \R$.
{\color{black}Our first result on the well-posedness and H\"older continuity property of Eq. \eqref{spde} is the following theorem, which implies that almost surely trajectories of $u$ are H\"older continuous. 
 }
\begin{tm}\label{wel-hol}
\begin{itemize}
 \item[(1)] Let Assumption \ref{ass-b} hold. {\color{black}}Assume that $u_0$ and $v_0$ are continuous and possess uniformly bounded $p$-th moments for $p\ge 2$. 
 Eq. \eqref{spde} has a unique mild solution $u=\{u(t,x):\ (t,x)\in [0,T]\times \R\}$ which is an adapted process satisfying
\begin{align*}
\sup_{(t,x)\in [0,T]\times \R} \E[\|u(t,x)\|^p]<\infty.
\end{align*}
If, in addition, Assumption \ref{ass-u0v0} holds, then there exists $C=C(p,T,H,d)$ such that for any $(t,x),(s,y)\in [0,T]\times \R,$
\begin{align}\label{mom-inc}
\|u(t,x)-u(s,y)\|_{\mathbb L^p} 
\le C \left(\triangle((t,x);(s,y))\right)^{\alpha\wedge H}.
\end{align}
\item[(2)]The estimate \eqref{mom-inc} is optimal in the sense that the reverse estimate of \eqref{mom-inc} holds with $\alpha\wedge H$ replaced by $H$, in compact intervals with $s$ and $t$ being sufficiently close, when $u_0=v_0=0$, $b=0$ and $\sigma=I_{d\times d}$.
\end{itemize}
\end{tm}

{\color{black}Base on the estimate \eqref{mom-inc} in Theorem \ref{wel-hol}, we further study the optimal H\"older continuity exponents of the exact solution of Eq. \eqref{spde}, which is in particular motivated by the research of hitting probabilities of systems of SPDEs.
 Given a random field $X=\{X(r)\}_{r \in B}$
with  $B$  being some Borel measurable subset of  $\mathbb{R}^{N}$, we say that $X$ hits a Borel set  $A \subseteq \mathbb{R}^{d}$ if 
 $\PP\{X(B) \cap A \neq \emptyset\}>0$, where $X(B)$  denotes the range of  $B$  under the random map  $r \mapsto X(r)$. In this case,  $A \subseteq \mathbb{R}^{d}$  is also called polar for  $X$;  otherwise $A$  is called nonpolar. Recently, there has been much progress on hitting probabilities of systems of SPDEs; see e.g. \cite{DKN07,NV09,DKN09,DKN13} for systems of SHEs and \cite{DS10,DS15} for systems of SWEs. Most of them in these literatures are driven by the space-time white noise or a
 noise whose spectral measure in space is given by a Riesz kernel.
 Following this line of investigation,
 based on Theorem \ref{wel-hol} and the criterion on hitting probabilities} developed in \cite{BLX09}, we obtain another
result Theorem \ref{0sigma} of this paper, which gives lower and upper bounds of hitting probabilities for the solution of Eq. \eqref{spde} for the linear case. 
 We say that a compact set $I\subset\R$ is non-trivial if there exist two non-empty closed intervals $I_1,\,I_2\subset\R$ such that $I_1\subset I\subset I_2$.
  \begin{tm}\label{0sigma}
 Assume that $u_0=v_0=0$, $b=0$, and $\sigma\in\R^{d\times d}$ is invertible.
 Let $I$ and $J$ be two non-trivial compact sets in $(0,T]$ and $\R$, respectively.  If  $A \subseteq \mathbb{R}^{d}$  is a Borel set,  then there exists $C=C(A,I,J,H,\sigma,d)>0$ such that
\begin{equation}\label{cd-new}
C^{-1}\operatorname{Cap}_{d-Q}(A) \le \mathbb{P}\left\{u(I\times J)\cap A \neq \emptyset\right\} \le C \mathscr{H}_{d-Q}(A),
\end{equation}
where $Q=3/H$ for SHE and $Q=2/H$ for SWE.
 \end{tm}

{\color{black}
For any set $\{y\}$ of single point with $y\in\R^d$, it holds that $\operatorname{Cap}_{d-Q}(\{y\})>0$ for $d<Q$ and $\mathscr{H}_{d-Q}(A)=0$ for $d>Q$. Thus,
it can be seen from Theorem \ref{0sigma} that  points are nonpolar for $u$ when $d<Q$, and points are polar for $u$ when $d>Q$. Here, the parameter $Q$ is the so-called critical dimension of hitting probabilities of the exact solution $u$.
To the best of our knowledge, Theorem 1.2 is the first result on the hitting probabilities of Eq. \eqref{spde} with $H<\frac{1}{2}$. This
   supplements the existing results about hitting probabilities of systems of SHEs or SWEs driven by a Gaussian noise that is white in time and fractional in space with Hurst index $H\ge\frac{1}{2}$ (see e.g. \cite{DS10,MT02}).  

   }

The rest of this paper is organized as follows. Some preliminaries including required estimates about Green's functions and stochastic integral for FBS are given in the next section.
In Section \ref{sec-wel}, we prove the well-posedness and H\"older continuity of Eq. \eqref{spde}. The optimality of H\"older exponents is presented in Section \ref{sec-opt}.
 Finally, based on the optimal H\"older continuity, we obtain the lower and upper bounds of the hitting probabilities for the solution in Section \ref{sec-hit}.

\section{Preliminaries}
We begin this section with introducing the following frequently used notations:
Without illustrated, all supremum with respect to $t$ (respectively, $x$ and $n$) denotes $\sup_{t\in [0,T]}$ (respectively, $\sup_{x\in \R}$ and $\sup_{n\in \N_+}$).
We also use $C$ (and $C_1$, $C_2$, etc.) to denote a generic constant which may change from line to line.

\subsection{Green's Functions and Mild Solutions}

We denote by $G_t(x)$ the Green's function of $Lu(t,x)=0$. 
It has the explicit form
\begin{align}\label{gre}
G_t(x)=
\begin{cases}
\frac1{\sqrt{4\pi t}}\exp (-\frac{|x|^2}{4t}), &\quad \text{for SHE};\\
\frac12 \mathbf 1_{\{|x|<t\}}, &\quad \text{for SWE},
\end{cases}
\end{align}
where $\mathbf 1$ denotes the indicator function.
The mild solution of Eq. \eqref{spde} is defined as an $\FFF_t$-adapted random field satisfying
\begin{align} \label{mild}
u(t,x)
&=\omega(t,x)+\int_0^t\int_{\R} G_{t-\theta}(x-\eta)b(u(\theta,\eta)) d\eta
d\theta \nonumber  \\
&\quad +\int_0^t\int_{\R} G_{t-\theta}(x-\eta)\sigma W(d\theta,d\eta).
\end{align}
Here,  $\omega(t,x)$ is the solution of the homogeneous equation with the same initial conditions as given in Section \ref{sec-1}. More precisely,
\begin{align}\label{wtx}
\omega(t,x)
=
\begin{cases}
\frac1{\sqrt{4\pi t}} \int_{\R} \exp(-\frac{|x-\eta|^2}{4t})u_0(\eta)d\eta, 
&\quad \text{for SHE};\\
\frac12 \left( u_0(x+t)-u_0(x-t)\right)
+\frac 12\int_{x-t}^{x+t} v_0(\eta)d\eta, &\quad \text{for SWE}.
\end{cases}
\end{align}

We need the following known results.

\begin{lm} \label{bjq}
\begin{enumerate}
\item[(1)]
Let $g:\R\rightarrow \R$ be a tempered function whose Fourier transform in $S'(\R)$ is a locally integrable function. 
Then for any $H\in (0,1/2)$,
\begin{align}\label{dif-cor}
c_H\int_{\R}|\F g(\xi)|^2 |\xi|^{1-2H} d\xi
&=C_H \int_{\R^2}
\frac{|g(x)-g(y)|^2}{|x-y|^{2-2H}} dxdy
\end{align}
when either one of the two integrals is finite, with $c_H$ given by \eqref{mu} and $C_H=\frac{H(1-2H)}{2}$.

\item[(2)] 
For any $\alpha\in (-1,1)$ and $t\ge0$,
\begin{align}\label{dif-cor-est}
\int_0^t \int_{\R} |\F G_\theta(\xi)|^2 |\xi|^\alpha d\xi d\theta
=
\begin{cases}
C_1 t^{\frac{1-\alpha}2}, &\quad \text{for SHE};\\
C_2 t^{2-\alpha}, &\quad \text{for SWE},
\end{cases}
\end{align}
for some constants $C_1$ and $C_2$ depending on $\alpha$.

\item[(3)]
For any $\alpha\in (-1,1)$ and any $h\in \R$, there exists $C=C(T)$ such that
\begin{align}\label{dif-cor-est-t}
\int_0^T \int_{\R}|\F G_{t+h}(\xi)-\F G_t(\xi)|^2 |\xi|^\alpha d\xi dt
&\le
\begin{cases}
Ch^{\frac{1-\alpha}2}, &\quad \text{for SHE};\\
Ch^{1-\alpha}, &\quad \text{for SWE}.
\end{cases}\\
\int_0^T \int_{\R} (1-\cos(\xi h)) 
 |\F G_t (\xi)|^2 |\xi|^\alpha d\xi dt
&\le C|h|^{1-\alpha}. \label{dif-cor-est-x}
\end{align}
\end{enumerate}
\end{lm}

\begin{proof}
We refer to \cite{BJQ15}, Proposition 2.8, Lemma 3.1, Lemma 3.5, and Lemma 3.4 for \eqref{dif-cor}, \eqref{dif-cor-est}, \eqref{dif-cor-est-t}, and \eqref{dif-cor-est-x}, respectively.
\end{proof}

\subsection{Fractional Noise and Stochastic Calculus}

In this subsection, we define the stochastic integral for the FBS $W$ determined by \eqref{fbs} with $H\le 1/2$ in one dimensional case, i.e., $d=1$. 
The arguments can be straightforward extended to general dimension $d\in \N_+$.

Recall that the fractional Brownian motion (FBM) in $\R$ with index $H\in (0,1)$ is a centered Gaussian process $B=\{B(x)\}_{x\in \R}$ with covariance
\begin{align*}
\E[B(x)B(y)]
=\int_{\R} \F\mathbf 1_{[0,x]}(\xi) \overline{\F \mathbf 1_{[0,y]}(\xi)} \mu(d\xi).
\end{align*}
In particular, the FBM with index $H=1/2$ coincides with the Brownian motion.
It is straightforward to verify that when $H\le 1/2$ then for any $\eta\in (1-H,\infty)$,
\begin{align*}
\int_{\R} \frac{\mu(d\xi)}{(1+|\xi|^2)^\eta}<\infty.
\end{align*}
On the other hand, when $H\in (1/2,1)$, the Fourier transform of 
$\mu$ is the locally integrable function $f(x)=H(2H-1)|x|^{2H-2}$. 
In this case, it is {\color{black}proportional} to the Gaussian random field determined by a Riesz kernel, which is investigated by a lot of authors (see, e.g., \cite{Dal99, DF98, DKN13, DS05, DS09, DS15, HHN14, SS02} and references therein).
However, when $H< 1/2$, the Fourier transform of $\mu$ is a genuine distribution and not a tempered measure.

The author in \cite{Jol10} shows that the domain of the Wiener integral for the FBM $B$ in $\R$ with index $H\in (0,1)$ is the completion of $\C_0^\infty(\R)$, the space of infinitely differentiable functions with compact support, with respect to the inner product 
\begin{align}\label{rkhs-b}
\<\varphi,\psi\>_{\HH}
:=\E[B(\varphi)B(\psi)]
=\int_{\R} \F \varphi(\xi) \overline{\F \psi(\xi)} \mu(d\xi),
\ \varphi,\psi\in \C_0^\infty(\R),
\end{align}
which coincides with the space of distribution $S\in \mathcal S'(\R)$, whose Fourier transform $\F S$ is a locally integrable function satisfying
\begin{align*}
\int_{\R} |\F S(\xi)|^2 \mu(d\xi)<\infty.
\end{align*}
There is a Hilbert space naturally associated with the FBM $B$.
Indeed, let $\HH$ be the completion of $\C_0^\infty(\R)$ with respect to the inner product defined by \eqref{rkhs-b}.
Therefore, $\HH$ is the reproducing kernel Hilbert space (RKHS) of the FBM $B$.

It is known that under assumption \eqref{mu}, $W=(W_1,\ldots,W_d)$ is an $\R^d$-valued FBS with temporal Hurst index $1/2$ and spatial Hurst index $H\le 1/2$
(including the standard Brownian sheet where $H=1/2$) 
 on the stochastic basis $(\Omega,\FFF,(\FFF_t)_{t\in [0,T]},\PP)$, i.e., for any $(t,x),(s,y)\in \R_+\times \R$ and $i,j\in \N_d$,
\begin{align}\label{fbs}
\E[W_i(t,x)W_j(s,y)]
=\delta_{ij}\left(t\wedge s\right)\frac{|x|^{2H}+|y|^{2H}-|x-y|^{2H}}{2}
\end{align}
with $t\wedge s:=\min\{t,s\}$.
In this level, $W_i, i=2,\ldots,d$, are independent copies of $W_1$, which is
a centered Gaussian family of random variables 
$\{W_1(\varphi):\ \varphi\in \mathcal \C_0^\infty([0,T]\times \R)\}$ defined on $(\Omega,\FFF,\PP)$, with covariance 
\begin{align}\label{rkhs-w1}
\E[W_1(\varphi)W_1(\psi)]
=\int_0^T \int_{\R} \F \varphi(t,\cdot)(\xi) \overline{\F \psi(t,\cdot)(\xi)} \mu(d\xi) dt
=:\<\varphi,\psi\>_{\HH_T}.
\end{align}
Here $\HH_T$, the completion of $\C_0^\infty([0,T]\times \R)$ with respect to the inner product defined by \eqref{rkhs-w1}, is the RKHS of the FBS $W_1$.
It can be identified with the homogenous Sobolev space of order $1/2-H$ of functions with values in $L^2(\R)$ (see, e.g., \cite{AMN01, HHLNT17(AOP)}).
Nevertheless, we will use another simpler characterization of 
$\HH_T$, in terms of \eqref{dif-cor} in Lemma \ref{bjq}, which is more suitable to our case:
\begin{align*}
\HH_T:=\left\{\phi\in \C_0^\infty([0,T]\times \R):\ \int_0^T \int_{\R^2} \frac{|\phi(s,y)-\phi(s,z)|^2}
{|y-z|^{2-2H}}dydzds<\infty \right\}.
\end{align*}

We have the following It\^o isometry.

\begin{tm}\label{ito-bdg}
Assume that $\varphi\in \HH_T$.
Then for any $t\in [0,T]$,
\begin{align}\label{ito}
&\E\left[\left| \int_0^t\int_{\R} \varphi(s,y) W_1(ds,dy)\right|^2 \right] 
=\int_0^t \int_{\R} |\F \varphi(s,\cdot)(y)|^2 \mu(dy) ds.
\end{align}
\end{tm}

\begin{proof}
See \cite[Theorem 2.7]{BJQ15} or \cite[Proposition 4.1]{Jol10}.
\end{proof}

\begin{rk}
\begin{enumerate}

\item
When $H=1/2$, the noise in Eq. \eqref{spde} reduces to the space-time white noise: 
$\mu(d\xi)=d\xi$. 
By Plancherel theorem, we have
\begin{align*}
\E\left[\left| \int_0^t\int_{\R} \varphi(s,y) W_1(ds,dy)\right|^2 \right] 
=\int_0^t \int_{\R} |\varphi(s,y)|^2 dyds.
\end{align*}
Thus all results, such as Theorem \ref{wel-hol}, of the paper hold for $H=1/2$.

\item
Since $\int_0^t\int_{\R} \varphi(s,y) W_1(ds,dy)$ is a centered Gaussian random variable, for any $p\ge 2$, there exists $C=C(p)$ such that for any $t\in [0,T]$,
{\small
\begin{align}\label{bdg}
\E\left[\left| \int_0^t\int_{\R} \varphi(s,y) W_1(ds,dy)\right|^p \right] 
=C\left( \int_0^t \int_{\R} |\F \varphi(s,\cdot)(y)|^2 \mu(dy) ds \right)^\frac p2.
\end{align}
}
\end{enumerate}
\end{rk}

\section{Well-posedness and H\"{o}lder Continuity}
\label{sec-wel}

We use the Picard iteration to prove the well-posedness of Eq. \eqref{spde} in this section, and then derive the sample-path H\"older continuity of the solution.

Given constants $\beta_1, \beta_2\in (0,1]$, denote by 
$\C_{\beta_1, \beta_2}:=\C_{\beta_1, \beta_2}([0,T]\times \R; \R^d)$ the set of functions $v: [0,T]\times \R\rightarrow \R^d$ which are temporally $\beta_1$-H\"older continuous and spatially $\beta_2$-H\"older continuous.
More precisely, for each compact subset $D\subset \R_+\times \R$, there is a finite constant $C$ such that for all $(t, x),(s, y)\in D$,
\begin{align*}
\|v(t,x)-v(s,y)\| \le C(|t-s|^{\beta_1}+|x-y|^{\beta_2}),
\end{align*} 
where we denote $\|v\|:=\sqrt{v_1^2+\cdots+v_d^2}$ for $v=(v_1,\ldots,v_d)\in \R^d$. 
Let 
\begin{align*}
\C_{\beta_1-, \beta_2-}:
=\bigcap_{0<\alpha_1<\beta_1}\bigcap_{0<\alpha_2<\beta_2} \C_{\alpha_1, \alpha_2}, \quad
\C_{\beta_1+, \beta_2+}:
=\bigcap_{\beta_1<\alpha_1\le 1}\bigcap_{\beta_2<\alpha_2\le 1} \C_{\alpha_1, \alpha_2}.
\end{align*} 
Similarly, one can also define $\C_{\beta_1}([0,T])$, $\C_{\beta_1\pm}([0,T])$ and, respectively,
$\C_{\beta_2}(\R)$, $\C_{\beta_2\pm}(\R)$ as the H\"older space of $v$ in temporal direction and spatial direction.

Define the Picard iteration scheme as
\begin{align}\label{pic}
\begin{split}
u^0(t,x):&=\omega(t,x);   \\
u^{n+1}(t,x):&=\omega(t,x)
+\int_0^t\int_{\R} G_{t-\theta}(x-\eta)b(u^n(\theta,\eta))d\eta d\theta  \\
&\quad +\int_0^t\int_{\R} G_{t-\theta}(x-\eta)\sigma W(d\theta,d\eta)
\end{split}
\end{align}
for $n\in\{0,1,2,\ldots\}$. {\color{black}We proceed to prove  Theorem \ref{wel-hol} (1), which gives the well-posedness and H\"older continuity of  Eq. \eqref{spde}.}

%

{\color{black}\textit{Proof of Theorem \ref{wel-hol} (1):}}
We start with verifying the uniform boundedness of the $p$-th moments of $u^n$.
The H\"{o}lder inequality and equality \eqref{bdg} imply the existence of $C=C(p,H,T,d)$ such that
\begin{align}\label{untx}\notag
\E[\|u^{n+1}(t,x)\|^p]  
\le& C\E[\|\omega(t,x)\|^p] +C\int_0^t \E\left[\left\|\int_{\R} G_{t-\theta}(x-\eta)b(u^n(\theta,\eta))d \eta \right\|^p \right] d \theta \\
&+ C\left(\int_0^t \int_{\R}
|\F G_\theta(x-\cdot)(\xi)|^2 |\xi|^{1-2H} d\xi d\theta\right)^\frac p2.
\end{align}
Applying the property of Fourier transform: $\F g(x+\cdot) (\xi)=e^{ix\xi}\F g(\xi)$, $\xi$-a.e., and the estimate \eqref{dif-cor-est}, the last term in the right hand side of \eqref{untx} is bounded uniformly.
By Minkowskii's inequality,
\begin{align*}
&\int_0^t \E\left[\left\|\int_{\R} G_{t-\theta}(x-\eta)b(u^n(\theta,\eta))d\eta \right\|^p \right] d\theta \\
\le& \int_0^t  \left(\int_{\R} G_{t-\theta}(x-\eta) d\eta \right)^p \left(\sup_{x} \E[\|b(u^n(\theta,x))\|^p] \right) d\theta \\
\le& C+C\int_0^t \left(\sup_{x}\E[\|u^n(\theta,x)\|^p]\right) g(t-\theta) d\theta,
\end{align*}
where 
\begin{align}\label{gtp}
g(t):=\left( \int_{\R} G_t(y) dy \right)^p
=\begin{cases}
1,&\quad \text{for SHE}   \\
t^p,&\quad \text{for SWE}
\end{cases}
\end{align}
is uniformly bounded in $[0,T]$.

The equalities \eqref{wtx} and \eqref{gtp} with $p=1$, in combination with the fact that $u_0$ and $v_0$ have uniformly bounded $p$-th moments, indicate that
\begin{align*}
\|\omega(t,x)\|_{\mathbb L^p}\le \sup_x\|u_0(x)\|_{\mathbb L^p}
+T\sup_x\|v_0(x)\|_{\mathbb L^p}\le C.
\end{align*}
Gathering the above estimates together, it follows that
\begin{align*}
\E[\|u^{n+1}(t,x)\|^ p]
\le C+C\int_0^t \left(\sup_{x}\E[\|u^n(\theta,x)\|^p]\right) d\theta.
\end{align*}
Therefore, if we set $M^n(t):=\sup_{x} \E[\|u^n(t,x)\|^ p]$, then
\begin{align*}
M^{n+1}(t)
\le C+C\int_0^t M^n(\theta) d\theta.
\end{align*}
Gronwall lemma yields that $\sum_{n=1}^\infty M^n(t)$ converges uniformly on $[0,T]$. In particular, $\sup_n \sup_t M^n(t)<\infty$, i.e., $\sup_n \sup_{t,x} \E[\|u^{n}(t,x)\|^p]<\infty$.

Next we prove that the scheme \eqref{pic} is convergent and the limit is the unique mild solution of Eq. \eqref{spde}.
It is clear that 
\begin{align*}
u^{n+1}(t,x)-u^n(t,x)
=\int_0^t\int_{\R} G_{t-\theta}(x-\eta)
\left( b(u^n(\theta,\eta))-b(u^{n-1}(\theta,\eta))\right) d\eta d\theta.
\end{align*}
Similar to the proof of moments' boundedness of $u^n$, we have
\begin{align*}
&\E[\|u^{n+1}(t,x)-u^n(t,x)\|^p]
\le C\int_0^t \left(\sup_{x}\E[\|u^n(\theta,x)-u^{n-1}(\theta,x)\|^ p]\right) d\theta,
\end{align*}
Define $H^n(t):=\sup_{x} \E[\|u^{n+1}(t,x)-u^n(t,x)\|^p]$. 
Then
\begin{align*}
H^n(t)
&\le C\int_0^t H^{n-1}(\theta) d\theta,
\end{align*}
from which we conclude by Gronwall lemma that 
$\sum_{n=1}^\infty H^n(t)$ converges uniformly on $[0,T]$. This shows that for each $t$ and $x$, $u^n(t,x)$ converges in $\mathbb L^p$ to $u(t,x)$, which is the mild solution of \eqref{spde} satisfying
$\sup_{t,x}\E[\|u(t,x)\|^p]<\infty$.
The same procedure yields the uniqueness of the mild solution of \eqref{spde}. 

Finally, we prove the H\"older continuity of the solution of Eq. \eqref{spde}.
Without loss of generality, assume that $0\le s<t\le T$.
For SHE, applying the semigroup property of $G$  
as well as the fact that $\int_{\R} G_t(\eta)d\eta=1$ and then using Jensen inequality, we obtain
\begin{align*}
&\E\left[\left\|\int_{\R} \left(G_t (x-\eta)-G_s (x-\eta)\right) u_0(\eta)d\eta\right\|^p\right]  \\
=&\E\left[\left\| \int_{\R}G_{t-s}(\eta)\left( \int_{\R} G_s(x-z) \left(u_0(z-\eta)-u_0(z)\right) dz\right) d\eta \right\|^p\right]  \\
\le& \int_{\R}G_{t-s}(\eta)  \int_{\R} G_s(x-z) \E[\|u_0(z-\eta)-u_0(z)\|^p]dz d\eta  \\
\le& \int_{\R}G_{t-s}(\eta) |\eta|^{p\alpha} d\eta  \le C|t-s|^{\frac{{p\alpha}}2}.
\end{align*}
Triangle inequality, Assumption \ref{ass-u0v0}, and Jensen inequality then yield
\begin{align*}
{\color{black}I_1:=}&\E[\|\omega(t,x)-\omega(s,y)\|^p]   \\
\le& C\E\left[\left\|\int_{\R} \left(G_t (x-\eta)-G_s (x-\eta)\right) u_0(\eta)d\eta\right\|^p\right] \\
&+C\E\left[\left\|\int_{\R} \left(G_s (x-\eta)-G_s (y-\eta)\right) u_0(\eta)d\eta\right\|^p\right]  \\
\le& C|t-s|^{\frac{{p\alpha}}2}
+C\int_{\R}G_s (\eta) \E\left[ \|u_0(x{\color{black}-}\eta)-u_0(y{\color{black}-}\eta)\|^p\right] d\eta  \\
\le& C\left(|t-s|^{\frac{{p\alpha}}2}+|x-y|^{p\alpha}\right).
\end{align*}

For SWE, by Assumption \ref{ass-u0v0}, {\color{black}the uniform boundedness of $\|v_0(x)\|_{\mathbb L^p}$, and the elementary inequality $\|a_1+a_2+a_3\|^p\le 3^{p-1}\left(\|a_1\|^p+\|a_2\|^p+\|a_3\|^p\right)$ for $a_i\in\R^d, i=1,2,3$, we have}
\begin{align*}
&\E[\|\omega(t,x)-\omega(s,y)\|^p]    \\
\le& C\E\left[\left\|u_0(x+t)+u_0(x-t)-u_0(y+s)-u_0(y-s)\right\|^p\right]  \\
&+C\E\left[\left\|\int_{\R} \left(G_t (x-\eta)-G_s (x-\eta)\right) v_0(\eta)d\eta\right\|^p\right] \\
&+C\E\left[\left\|\int_{\R} \left(G_s (x-\eta)-G_s (y-\eta)\right) v_0(\eta)d\eta\right\|^p\right]  \\
\le& C\left(|t-s|^{p\alpha} +|x-y|^{p\alpha}\right).
\end{align*}
Therefore, we have proved 
\begin{align*}
\E[\|\omega(t,x)-\omega(s,y)\|^p]
\le C\left(\triangle((t,x);(s,y))\right)^{p\alpha}.
\end{align*}

Applying equality \eqref{bdg} and It\^o isometry \eqref{ito}, we obtain
\begin{align*}
{\color{black}I_2:}&=\E\left[\left\|\int_0^t\int_{\R} G_{t-\theta}(x-\eta) \sigma W(d\theta,d\eta)
-\int_0^s \int_{\R} G_{s-\theta}(y-\eta)\sigma W(d\theta,d\eta) \right\|^p\right] \\
&\le C\left(\int_s^t\int_{\R} |\F G_{t-\theta}(\xi)|^2 |\xi|^{1-2H}d\xi d\theta \right)^\frac p2  \\
&\quad +C\left(\int_0^s \int_{\R} 
|\F \left(G_{t-\theta}(x-\cdot)-G_{s-\theta}(x-\cdot)\right) (\xi) |^2 |\xi|^{1-2H} d\xi d\theta\right)^\frac p2 \\
&\quad +C\left(\int_0^s \int_{\R} 
|\F \left(G_{s-\theta}(x-\cdot)-G_{s-\theta}(y-\cdot)\right) (\xi)|^2
|\xi|^{1-2H}d\xi d\theta \right)^\frac p2.
\end{align*}
The fact that 
$\F g(x+\cdot)(\xi)=e^{i x\xi}\F g(\xi)$, $\xi$-a.e., and estimates \eqref{dif-cor-est}--\eqref{dif-cor-est-x} imply 
\begin{align*}
{\color{black}I_2}\le&C\left(\int_0^{t-s} \int_{\R} |\F G_\theta(\xi)|^2 |\xi|^{1-2H} d\xi d\theta \right)^\frac p2 \\
&\quad +C\left(\int_0^s \int_{\R} 
|\F G_{t-s+\theta}(\xi)-\F G_{\theta}(\xi)|^2
|\xi|^{1-2H} d\xi d\theta \right)^\frac p2  \\
&\quad +C\left(\int_0^s \int_{\R} 
\left( 1-\cos(\xi(x-y)) \right)|\F G_{\theta}(\xi)|^2
|\xi|^{1-2H} d\xi d\theta  \right)^\frac p2  \\
\le& C\left(\triangle((t,x);(s,y))\right)^{Hp}.
\end{align*}
By {\color{black} the elementary inequality $\|a_1+a_2\|^p\le 2^{p-1}\left(\|a_1\|^p+\|a_2\|^p\right)$, $a_i\in\R^d$, $i=1,2,$}  and change of variables, we have
\begin{align*}
{\color{black}I_3}:=&\E\left[\left\|\int_0^t \int_{\R} G_{t-\theta}(x-\eta) b(u(\theta,\eta)) d\eta d\theta
-\int_0^s \int_{\R} G_{s-\theta}(y-\eta) b(u(\theta,\eta)) d\eta d\theta \right\|^p\right]\\
\le& C\E\left[\left\|\int_0^{t-s} \int_{\R} G_{t-\theta}(x-\eta) b(u(\theta,\eta)) d\eta d\theta\right\|^p\right]\\
&+C\E \left[\left\|\int_{t-s}^{t} \int_{\R} G_{t-\theta}(x-\eta)  \left(b(u(\theta,\eta))-b\left( u(\theta-(t-s),\eta-(x-y)) \right)\right) d\eta d\theta \right\|^p\right].
\end{align*}
Since $b$ is Lipschitz continuous and $\E[\|u(t,x)\|^p]$ is uniformly bounded,
\begin{align*}
&\E \left[\left\|\int_0^{t-s} \int_{\R} G_{t-\theta}(x-\eta) b(u(\theta,\eta)) d\eta d\theta\right\|^p\right] 
\le C|t-s|^p.
\end{align*}
Applying H\"older's inequality and change of variables, we obtain
\begin{align*}
&\E \left[\left\|\int_{t-s}^{t} \int_{\R} G_{t-\theta}(x-\eta) 
\left(b(u(\theta,\eta))-b\left( u(\theta-(t-s),\eta-(x-y)) \right)\right) d\eta d\theta \right\|^p\right]\\
 \le& C\int_0^s 
\sup_{\eta} \E \left[ \left\|u(\theta+(t-s),\eta)-u(\theta,\eta-(x-y))\right\|^p\right] d\theta.
\end{align*}
The above estimates {\color{black}on $I_1$, $I_2$, and $I_3$, together with \eqref{mild},} yield
\begin{align*}
&\E [\|u(t,x)-u(s,y)\|^p] { \color{black}\le 3^{p-1}(I_1+I_2+I_3)}\\
\le &C\left(\triangle((t,x);(s,y))\right)^{p(\alpha\wedge H)} 
+C\int_0^s 
\sup_{\eta} \E \left[ \left|u(\theta+(t-s),\eta)-u(\theta,\eta-(x-y))\right|^p\right] d\theta.
\end{align*}
We conclude by Gronwall lemma that 
\begin{align*}
\E [\|u(t,x)-u(s,y)\|^p]
&\le C\left(\triangle((t,x);(s,y))\right)^{p(\alpha\wedge H)},
\end{align*}
which completes the proof of \eqref{mom-inc}.
\qed\\
As a result of  Kolmogorov continuity theorem and \eqref{mom-inc}, $u$ has a version which is in $\C_{\frac{\alpha\wedge H}2-, (\alpha\wedge H)-}$, a.s.
\begin{rk}\label{hol-com}
For Eq. \eqref{spde} with $d=1$ and vanishing drift driven by affine noise, i.e., $\sigma=\sigma_1 u+\sigma_2$ with $\sigma_1,\sigma_2\in \R$, the authors in \cite[Theorem 1]{BJQ16} have proved that there exist a constant $h_0\in (0,1)$ such that \eqref{mom-inc} holds for all
$|t-s|\le h_0$ and $|x-y|\le h_0$.
In our case, there is no restriction on $h_0$.
\end{rk}

\section{Optimality of H\"older Exponents}
\label{sec-opt}

In this section, we investigate the optimality of the estimate \eqref{mom-inc} for
\begin{align*}
u(t,x)=\int_0^t\int_{\R} G_{t-\theta}(x-\eta) W(d\theta,d\eta),\quad
(t,x)\in [0,T]\times \R,
\end{align*}
which is the solution of Eq. \eqref{spde} with $u_0=v_0=0$, $b=0$, and 
$\sigma=I_{d\times d}$.

Our main purpose in this section is to prove 
\begin{tm}\label{hol-opt}
Assume that $u_0=v_0=0$, $b=0$, and $\sigma=I_{d\times d}$.
\begin{enumerate}
\item[(1)]
Fix $t\in {\color{black}(}0,T]$ and a compact interval $J$.
There exists $C>0$ such that  
\begin{align}\label{hol-opt-x}
\|u(t,x)-u(t,y)\|_{\mathbb L^2} \ge C|x-y|^H, \quad x,y\in J.
\end{align}
Consequently, a.s., the mapping $x\mapsto u(t,x)$ is not in 
$\C_{H+}(\R)$.

\item[(2)]
Fix $x\in \R$ and $t_0\in (0,T]$.
There exists $C>0$ such that 
\begin{align}\label{hol-opt-t}
\|u(t,x)-u(s,x)\|_{\mathbb L^2}
\ge
\begin{cases}
C|t-s|^\frac H2,&\quad \text{for SHE};\\
C|t-s|^H,&\quad \text{for SWE}{\color{black},}
\end{cases}
\end{align}
for all $t,s\in [t_0,T]$ with $|t-s|$ sufficiently small.
Consequently, a.s., the mapping $t\mapsto u(t,x)$ is not in 
$\C_{\frac H2+}([0,T])$ for SHE or $\C_{H+}([0,T])$ for SWE.
\end{enumerate}
\end{tm}

\begin{proof}
Since $W_1,\ldots,W_d$ are independent identically distributed, we have
\begin{align*}
\|u(t,x)-u(s,y)\|_{\mathbb L^2}=\sqrt d\|u_1(t,x)-u_1(s,y)\|_{\mathbb L^2},\quad (t,x),(s,y)\in[0,T]\times\R.
\end{align*}

Without loss of generality, we assume that $t=1$.
For any $x\in \R$, set $R(x):=\E[u_1(1,0)u_1(1,x)]$.
It is clear that $\E[|u_1(1,x)-u_1(1,y)|^2]=2\left( R(0)-R(x-y) \right)$, and then to prove \eqref{hol-opt-x} it suffices to show that for any $x\in \R$, there exists $C>0$ such that 
\begin{align*}
R(0)-R(x)\ge C|x|^{2H}.
\end{align*}

By It\^o isometry \eqref{ito} and simple calculations, we have 
\begin{align*}
R(0)-R(x)
&=c_H\int_0^1 \int_{\R} \left(1-\frac{e^{ i \xi x}+e^{\mathbf -i \xi x}}2\right) |\F G_{t-s}(\xi)|^2
|\xi|^{1-2H} d\xi ds \\
&=\begin{cases}
c_H \int_{\R} \frac{1-\cos(\xi x)}{|\xi|^{2H+1}}
\big(1-e^{-|\xi|^2}\big)d\xi,
&\quad \text{for SHE};\\
\frac12 c_H \int_{\R} \frac{1-\cos(\xi x)}{|\xi|^{2H+1}}
\big(1-\frac{\sin(2|\xi|)}{2|\xi|} \big)d\xi,
&\quad \text{for SWE}.
\end{cases}
\end{align*}
The integrands above are both non-negative, and for $|\xi|>1$,
\begin{align*}
1-e^{-\frac{|\xi|^2}2}>1-e^{-\frac12},\quad 
1-\frac{\sin(2|\xi|)}{2|\xi|} \ge \frac12.
\end{align*}
Applying change of variables $y=|x| \xi$, we obtain
\begin{align*}
R(0)-R(x)
\ge C|x|^{2H} \int_{|y|\ge |x|} \frac{1-\cos y}{|y|^{2H+1}} dy.
\end{align*}
Since $x\in J$ with $J$ compact, the last integral is bounded below by a positive constant.
This proves \eqref{hol-opt-x}.

We now turn to proof of \eqref{hol-opt-t} for $t_0\leq s<t\leq T$.
In this situation, by It\^o isometry \eqref{ito},
\begin{align*}
&\quad\E\left[\left|u_1(t,x)-u_1(s,x) \right|^2\right]  \\
&=\E\left[\left|\int_s^t \int_{\R}G_{t-\theta}(x-\eta) W_1(d\theta,d\eta)\right|^2 \right] \\
&\quad +\E\left[\left|\int_0^s \int_{\R}\left( G_{t-\theta}(x-\eta)-G_{s-\theta}(x-\eta) \right) W_1(d\theta,d\eta)\right|^2 \right] \\
&=C_H \int_s^t \int_{\R} |\F G_{t-\theta}(\xi)|^2 |\xi|^{1-2H} d\xi d\theta\\
&\quad +C_H \int_0^s \int_{\R} |\F G_{t-\theta}(\xi)-\F G_{s-\theta}(\xi)|^2 |\xi|^{1-2H} 
 d\xi d\theta \\
&=:C_H I_1+C_H I_2.
\end{align*}
Applying \eqref{dif-cor-est}, we obtain an estimate of $I_1$:
\begin{align*}
I_1
=\begin{cases}
C_1 (t-s)^H, &\ \text{for SHE};\\
C_2 (t-s)^{2H+1}, &\ \text{for SWE}.
\end{cases}
\end{align*}
For $I_2$, we estimate separately for SHE and SWE.

For SHE, simple calculations yield
\begin{align*}
&\quad\int_0^s |\F G_{t-\theta}(\xi)-\F G_{s-\theta}(\xi)|^2  d\theta =(1-e^{-(t-s)|\xi|^2})^2
\frac{1-\exp\left(-2s|\xi|^2\right)}{2|\xi|^2}.
\end{align*}
Then by change of variables $y=\sqrt{t-s}\,\xi$, we obtain
\begin{align*}
I_2
&\ge \int_{|\xi|\ge \frac1{\sqrt{t-s}}}
(1-e^{-(t-s)|\xi|^2})^2
\frac{1-\exp\left(-s|\xi|^2\right)}{2|\xi|^{2H+1}} d\xi \\
&\ge \left(1-\exp\left(-\frac{t_0}{t-s} \right) \right) \int_{|\xi|\ge \frac1{\sqrt{t-s}}}
\frac{(1-e^{-(t-s)|\xi|^2})^2 }{2|\xi|^{2H+1}} d\xi \\
&=(t-s)^H \left(1-\exp\left(-\frac{t_0}{t-s} \right) \right) \int_{|y|\ge 1}  \frac{(1-e^{-|y|^2})^2}{2|y|^{2H+1}} dy \\
&\ge (t-s)^H \left(1-\exp\left(-\frac{t_0}{t-s} \right) \right) 
\frac{(1-e^{-1})^2}{2H} .
\end{align*}
For sufficient close $s$ and $t$, say $t-s\leq \frac{t_0}{\ln 2}$, one has $1-\exp(-\frac{t_0}{t-s})\ge \frac12$, and then
\begin{align*}
I_2
\ge \frac{(1-e^{-1})^2}{4H} (t-s)^H.
\end{align*}

For SWE, we have
\begin{align*}
&\quad\int_0^s |\F G_{t-\theta}(\xi)-\F G_{s-\theta}(\xi)|^2  d\theta\\
& =\frac{s(1-\cos((t-s)|\xi|)}{|\xi|^2}
+\frac{(1-\cos((t-s)|\xi|))\sin((t+s)|\xi|)}{2|\xi|^3} \\
&\quad +\frac{\sin(2(t-s)|\xi|)-2\sin((t-s)|\xi|)}{4|\xi|^3}.
\end{align*}
Then by using change of variables $y=(t-s)\xi$, we obtain
\begin{align*}
I_2
&\ge t_0 (t-s)^{2H} \int_{|y|\ge 1}
\frac{1-\cos(|y|)}{|y|^{2H+1}} dy
-\frac{(t-s)^{2H+1}}2 \int_{|y|\ge 1}\frac1{|y|^{2H+1}} dy \\
&\quad -\frac{3(t-s)^{2H+1}}4\int_{|y|\ge 1}
\frac1{|y|^{2+2H}} dy  \\
&=C_1(t-s)^{2H}-\frac{5H+1}{2H(2H+1)}(t-s)^{2H+1},
\end{align*}
where $C_1=t_0 \int_{|y|\ge 1}
\frac{1-\cos(|y|)}{|y|^{2H+1}} dy$ is positive and bounded.
Then when $t-s\leq \frac{H(2H+1)C_1}{5H+1}$, one has 
\begin{align*}
I_2 \ge \frac{C_1}2(t-s)^{2H}.
\end{align*}
This proves \eqref{hol-opt-t}.

Optimal H\"{o}lder exponent is crucial to verify the optimality of convergence rate of numerical schemes (see, e.g., \cite{CHL17(SINUM), CHL18(IMA)}). To concern the absence of H\"older continuity, we need the following Fernique-type theorem which says that an a.s. uniformly bounded centered, Gaussian process has bounded moments. 

\begin{lm}\label{gau}
For a centered Gaussian process $X=\{X(t),t\in\mathcal T\}$,
\begin{align}
\PP\left\{\sup_{t\in \mathcal T} X(t)<\infty \right\}=1
\Longleftrightarrow
\E\left[\exp\left(\alpha|\sup_{t\in \mathcal T} X(t)|^2\right) \right]<\infty
\end{align}
for sufficiently small $\alpha>0$.
\end{lm}

Proof:
See \cite[Theorem 3.2]{Adl90}.
\qed\\

Now we consider the absence of H\"older continuity of $u$ through Lemma \ref{gau}.
It is clear that $u(t,x)-u(s,x)$ and $u(t,x)-u(t,y)$ are both centered, Gaussian for any $0\le s<t\le T$ and $x,y\in \R$.
We only give details for the space variable, while the arguments are available for the time variable.

Suppose that for a fixed $t\in (0,T]$, the sample-paths $x\mapsto u(t,x)$ are $\gamma$-H\"older continuous for some $\gamma>H$.
Then for any compact interval $J$, there exists $C(\omega)\in (0,\infty)$ such that 
\begin{align*}
\sup_{x,y\in J,x\neq y} \frac{u(t,x)-u(t,y)}{|x-y|^\gamma}\le C(\omega).
\end{align*}
This yields that the centered Gaussian process
\begin{align*}
\left\{\frac{u(t,x)-u(t,y)}{|x-y|^\gamma},\quad x,y\in J,\ x\neq y\right\}
\end{align*}
is finite a.s., from which we conclude by Lemma \ref{gau} that 
\begin{align*}
\E\left[\sup_{x,y\in J,x\neq y} 
\left|\frac{u(t,x)-u(t,y)}{|x-y|^\gamma}\right|^2\right]<\infty.
\end{align*}
In particular, there would exist a finite $C>0$ such that 
\begin{align*}
\E\left[|u(t,x)-u(t,y)|^2\right]\le C|x-y|^{2\gamma},
\end{align*}
which contracts \eqref{hol-opt-x}.
\end{proof}

Now we can prove Theorem \ref{wel-hol}.

\begin{proof}[Proof of Theorem \ref{wel-hol}]
Theorem \ref{wel-hol} follows from the \textit{proof of Theorem \ref{wel-hol} (1)} in Section \ref{sec-wel}  and Theorem \ref{hol-opt}.
\end{proof}
{\color{black} In the literature, when considering the lower bound for H\"older continuity exponents of a random field determined by an SPDE, one usually restrict the study to the special linear case (see e.g., \cite[Theorem 5.1]{DS09}, \cite[Theorem 6.2]{HHN14}). This linear case is simpler, and is sufficient to determine the optimal H\"older continuity exponents. For the case of  $b\neq constant$, the exact expression of $\tilde u_1(t,x):=\int_0^t\int_{\R} G_{t-\theta}(x-\eta)b(u^n(\theta,\eta))d\eta d\theta$ is unknown, 
which brings difficulty to extend the result in Section \ref{sec-opt}
to the solution $u(t,x)$ in Theorem \ref{wel-hol}, even for the case of $u_0=v_0=0$. 
}

\section{Hitting Probability}\label{sec-hit}
For any Borel sets  $F \subset \R^{d}$,  we define  $\mathcal P(F)$  to be the set of all probability measures with compact support in  $F$.  For $\mu \in \mathcal P(\R^{d})$,  denote by $I_{\beta}(\mu)$ the  $\beta$-dimensional energy of  $\mu$;  that is,
$$I_{\beta}(\mu):=\iint \mathrm{K}_{\beta}(\|x-y\|) \mu(d x) \mu(d y),$$
where  $\|x\|$  denotes the Euclidean norm of  $x \in \R^{d}$.  Here and throughout,
\begin{equation*}
\mathrm{K}_{\beta}(r):=\left\{\begin{array}{ll}
r^{-\beta} & \text { if } \beta>0, \\
\log \left(\frac{e}{r\wedge1}\right) & \text { if } \beta=0, \\
1 & \text { if } \beta<0.
\end{array}\right.
\end{equation*}
For any  $\beta \in\R$ and a Borel set  $F \subset\R^{d}$, 
$\operatorname{Cap}_{\beta}(F)$  denotes the  $\beta$-dimensional Bessel--Riesz capacity of  $F$;  that is,
$$\operatorname{Cap}_{\beta}(F):=\left[\inf _{\mu \in \mathcal{P}(F)} I_{\beta}(\mu)\right]^{-1},$$ where $1/\infty:=0.$ Given  $\beta \ge 0$,  the  $\beta$-dimensional Hausdorff measure of  $F$  is defined by
$$\mathscr{H}_{\beta}(F)=\lim _{\epsilon \rightarrow 0^{+}} \inf \left\{\sum_{i=1}^{\infty}\left(2 r_{i}\right)^{\beta}: F \subset \bigcup_{i=1}^{\infty} B\left(x_{i}, r_{i}\right), \sup _{i \geq 1} r_{i} \leq \epsilon\right\},$$
where  $B(x, r)$  denotes the open Euclidean ball of radius  $r>0$  centered at  $x \in \R^{d}$.  When  $\beta<0$,  we define  $\mathscr{H}_{\beta}(F):=\infty$.

Based on the optimal H\"older continuity, we show the lower and upper bounds for hitting probabilities of Eq. \eqref{spde} in this section. For this purpose,
we begin with introducing the following criterion on the hitting probability of a general Gaussian random field (see \cite{BLX09}, Theorem 2.1). 
\begin{tm}\label{Gausshit}
Let $B=[a,b]:=\prod_{j=1}^N[a_j,b_j]$ $(a_j<b_j)$ be an interval or a rectangle in $\R^N$ and $X=\left\{X(t),\,t \in \mathbb{R}^{N}\right\}$  be an $\R^d$-valued Gaussian random field with
$X(t)=\left(X_1(t),\cdots,X_d(t)\right)$ with  independent and identically distributed coordinate processes  $X_{1}, \ldots, X_{d}$. Assume that  there exist positive and finite constants  $c_{1}$, $c_{2}$, $c_{3}$, $c_{4}$ such that 
\begin{itemize}
\item[(C1)] 
 $\mathbb{E}\left[|X_{1}(t)|^{2}\right] \ge c_{1},$ for all  $t \in B$,
and there exists $\tau=(\tau_1,\ldots,\tau_N)\in(0,1)^N$  such that for all $s,\, t \in B,$
$$c_{2} \sum_{j=1}^{N}\left|s_{j}-t_{j}\right|^{2 \tau_{j}}  \le \mathbb{E}\left[\left|X_{1}(s)-X_{1}(t)\right|^{2}\right]  \le c_{3} \sum_{j=1}^{N}\left|s_{j}-t_{j}\right|^{2 \tau_{j}}.$$

\item[(C2)] There exists a constant  $c_{4}>0$  such that for all  $s,\, t \in B,$ 
$$\operatorname{Var}\left(X_{1}(t) | X_{1}(s)\right) \ge c_{4} \sum_{j=1}^{N}\left|s_{j}-t_{j}\right|^{2 \tau_{j}}.$$
\end{itemize}
Here,  $\operatorname{Var}\left(X_{1}(t) | X_{1}(s)\right) $ denotes the conditional variance of  $X_{1}(t)$  given  $X_{1}(s)$. Then there exist positive constants $c_5$, $c_6$ such that for every Borel set $A$ in $\mathbb{R}^{d}$,
$$c_{5} \operatorname{Cap}_{d-Q}(A)  \le \mathbb{P}\left\{X(B)\cap A \neq \emptyset\right\}  \le c_{6} \mathscr{H}_{d-Q}(A),$$
where $Q:=   \sum_{j=1}^{N} 1 / \tau_{j}$.
\end{tm}

We proceed to apply Theorem \ref{Gausshit} to the random field $u=\{u(t,x),(t,x)\in[0,T]\times\R\}$ defined by \eqref{spde} with $b\equiv0$ and $\sigma=I_{d\times d}$. More precisely, we will show in Lemmas \ref{Condi1} and \ref{Condi2}
that the real-valued random field $V=\{V(t,x),(t,x)\in[0,T]\times\R\}$ defined by
\begin{equation}\label{Vtx}
V(t,x)=\int_0^t\int_{\R}G_{t-s}(x,y)W_1(ds,dy)
\end{equation}
satisfies the above conditions (C1) and (C2).
For our case, we define 
\begin{equation*}\label{tau}
\tau:=\begin{cases}
(H/2,H), &\quad \text{for SHE};\\
(H,H), &\quad \text{for SWE}.
\end{cases}
\end{equation*}
\begin{lm}\label{Condi1}
Let $V$ be defined by \eqref{Vtx}, $M>0$, and $t_0\in(0,T)$. Then there exists $C_0=C_0(t_0,H)$ such that for any $t\in[t_0,T]$ and $x,y\in[-M,M]$,
\begin{align}\label{vtx^2}
\E\left[|V(t,x)|^2\right]\ge C_0.
\end{align}
Moreover, 
there exist  $C_i=C_i(t_0,T,M,H),\,i=1,2$, such that for any $(t,x),(s,y)\in[t_0,T]\times[-M,M]$,
\begin{align}\label{Vsytx}
C_1\left(|s-t|^{2\tau_1}+|x-y|^{2\tau_2}\right) \le \mathbb{E}\left[\left|V(s,y)-V(t,x)\right|^{2}\right] \le C_2\left(|s-t|^{2\tau_1}+|x-y|^{2\tau_2}\right).
\end{align}

\end{lm}
Proof:
By Theorem \ref{ito-bdg} and Lemma \ref{bjq} (2), there exist positive constants $A_1=A_1(H)$ and $A_2=A_2(H)$ such that 
\begin{equation}\label{Vtxp}
\E\left[|V(t,x)|^2\right]=
\begin{cases}
A_1 t^{H}, &\quad \text{for SHE};\\
A_2 t^{1+2H}, &\quad \text{for SWE},
\end{cases}
\end{equation}
which proves \eqref{vtx^2}. The upper bound of \eqref{Vsytx} follows from Theorem \ref{wel-hol}. 
By Theorems \ref{wel-hol} and \ref{hol-opt}, there exist $c_1,\,c_2,\,c_3,\,c_4>0$ such that  for any $t\in[t_0,T]$,
 $$c_1|x-y|^{2\tau_2}\le \mathbb{E}\left[\left|V(t,y)-V(t,x)\right|^{2}\right]\le c_2|x-y|^{2\tau_2},\,\forall\,x,y\in[-M,M],$$
  and that
  for any $x\in[-M,M]$,
$$c_3|t-s|^{2\tau_1}\le \mathbb{E}\left[\left|V(t,x)-V(s,x)\right|^{2}\right]\le c_4|t-s|^{2\tau_1},\,\forall\,t,s\in[t_0,T].$$ 
Based on the above arguments, for the low bound of \eqref{Vsytx}, it suffices to follow the approach of the proof of Lemma 3.1 in \cite{NV09} or Proposition 4.1 in \cite{DS10}. We briefly show the proof.

If $|x-y|^{2\tau_2}\ge \frac{4c_4}{c_1}|t-s|^{2\tau_1}$, then we have
\begin{align*}
\E\left[|V(t, x)-V(s, y)|^{2}\right] &\ge \frac{1}{2} \E\left[|V(t, x)-V(t, y)|^{2}\right]-\E\left[|V(t, y)-V(s, y)|^{2}\right] \\
&\ge \frac{c_{1}}{2} |x-y|^{2\tau_2}-c_{4}|t-s|^{2\tau_1} \ge \frac{c_{1}}{4}|x-y|^{2\tau_2} \\
&\ge \frac{c_{1}}{4}\left(\frac{|x-y|^{2\tau_2}}{2}+ \frac{2c_4}{c_1}|t-s|^{2\tau_1}\right) \\
&\ge \min\left\{\frac{c_1}{8},\frac{c_4}{2}\right\}\left(|s-t|^{2\tau_1}+|x-y|^{2\tau_2}\right).
\end{align*}
Similarly, if $|t-s|^{2\tau_1}\ge \frac{4c_2}{c_3}|x-y|^{2\tau_2}$, we also have
\begin{align*}
\E\left[|V(t, x)-V(s, y)|^{2}\right] 
\ge  \min\left\{\frac{c_3}{8},\frac{c_2}{2}\right\}\left(|s-t|^{2\tau_1}+|x-y|^{2\tau_2}\right).
\end{align*}
For the case $\frac{c_1}{4c_4}|x-y|^{2\tau_2}\le|t-s|^{2\tau_1}\le \frac{4c_2}{c_3}|x-y|^{2\tau_2},$ it suffices to show that
\begin{equation}\label{A1+A2}
\E\left[|V(t, x)-V(s, y)|^{2}\right] 
\ge c_7|s-t|^{2\tau_1}.
\end{equation}
for some $c_7=c_7(H)>0$.
Indeed, the left hand of \eqref{A1+A2} is equal to
\begin{align*}
A_1(s,t;x,y)+A_2(s,t;x,y)
\end{align*}
with 
\begin{align*}
A_1(s,t;x,y)&:=\int_0^s\int_{\R}|\F \{G_{t-r}(x,\cdot)-G_{s-r}(y,\cdot)\}(\xi)|^2\mu(d\xi)dr,\\
A_2(s,t;x,y)&:=\int_s^t\int_{\R}|\F G_{t-r}(x,\cdot)(\xi)|^2\mu(d\xi)dr= \begin{cases}
C_1 |t-s|^{H}, &\quad \text{for SHE};\\
C_2 |t-s|^{1+2H}, &\quad \text{for SWE},
\end{cases}
\end{align*}
for some $C_1,C_2>0$, thanks to
Theorem \ref{ito-bdg} and Lemma \ref{bjq} (2). For SHE, the proof of \eqref{A1+A2} is finished since $2\tau_1=H$ and $A_1(s,t;x,y)$ is nonnegative. For SWE,  by taking $k-\beta=2H-1$ with $k=1$, i.e., taking $\beta=2-2H$
 in the Step 2 of the proof of Proposition 4.1 in \cite{DS10}, we obtain that
 $$A_2(s,t;x,y)\ge C_2 |t-s|^{2H}.$$
The proof is completed.
\qed
\begin{lm}\label{Condi2}
Let $V$ be defined by \eqref{Vtx}, $t_0\in(0,T)$. Then there exists $C=C(t_0,H)$ such that for any $s,t\in[t_0,T]$ and $x,y\in\R$,
\begin{align*}
\operatorname{Var}\left(V(t,x)|V(s,y)\right)\ge C\left(|s-t|^{2\tau_1}+|x-y|^{2\tau_2}\right).
\end{align*}
\end{lm}
Proof:
The following fact will be used (see e.g. \cite{BLX09}): if  $(U, V)$  is a centered Gaussian vector, then
$$\operatorname{Var}(U| V)=\frac{\left(\rho_{U, V}^{2}-(\sigma_{U}-\sigma_{V})^{2}\right)((\sigma_{U}+\sigma_{V})^{2}-\rho_{U, V}^{2})}{4 \sigma_{V}^{2}},$$
where  $\rho_{U, V}^{2}=\mathbb{E}\left[|U-V|^{2}\right], \sigma_{U}^{2}=\mathbb{E}\left[|U|^{2}\right]$  and  $\sigma_{V}^{2}=\mathbb{E}\left[|V|^{2}\right]$. 

For any $s,t\in(0,T]$ and $x,y\in\R$, denote $\gamma_{t,x;s,y}^2:=\E[|V(t,x)-V(s,y)|^2]$ and $\sigma_{t,x}^2:=\E[|V(t,x)|^2]$. By \eqref{Vtxp}, it suffices to show 
\begin{align}\label{gamtxsy}\nonumber
&(\gamma_{t,x;s,y}^2-(\sigma_{t,x}-\sigma_{s,y})^2)((\sigma_{t,x}+\sigma_{s,y})^2-\gamma_{t,x;s,y}^2)\\\ge &C\left(|s-t|^{2\tau_1}+|x-y|^{2\tau_2}\right).
\end{align}
By \eqref{vtx^2} and the right hand side of \eqref{Vsytx}, we have
$$(\sigma_{t,x}+\sigma_{s,y})^2-\gamma_{t,x;s,y}^2\ge 2C_0-C_2\left(|s-t|^{2\tau_1}+|x-y|^{2\tau_2}\right),$$
which implies that $(\sigma_{t,x}+\sigma_{s,y})^2-\gamma_{t,x;s,y}^2$ is bounded below by the positive constant $C_0$ when $|s-t|^{2\tau_1}+|x-y|^{2\tau_2}\le C_0/C_2$.
From \eqref{vtx^2}, it follows that
$$|\sigma_{t,x}-\sigma_{s,y}|=\frac{|\sigma_{t,x}^2-\sigma_{s,y}^2|}{\sigma_{t,x}+\sigma_{s,y}}\le \frac{|\sigma_{t,x}^2-\sigma_{s,y}^2|}{2\sqrt C_0},$$ 
where for SHE,
 $ |\sigma_{t,x}^2-\sigma_{s,y}^2|=A_1|t^H-s^H|\le A_1|t-s|^H,$
and for SWE,  $ |\sigma_{t,x}^2-\sigma_{s,y}^2|=A_2|t^{2H+1}-s^{2H+1}|\le A_2(2H+1)T^{2H}|t-s|,$
in view of \eqref{Vtxp}. This together with  the left hand of \eqref{Vsytx} indicates that for SHE,
\begin{align*}
\gamma_{t,x;s,y}^2-(\sigma_{t,x}-\sigma_{s,y})^2&\ge C_1\left(|s-t|^{H}+|x-y|^{2H}\right)-\frac{A_1^2}{4C_0}|t-s|^{2H}\\
&\ge \frac{C_1}{2}\left(|s-t|^{2\tau_1}+|x-y|^{2\tau_2}\right),
\end{align*}
and for SWE,
\begin{align*}
\gamma_{t,x;s,y}^2-(\sigma_{t,x}-\sigma_{s,y})^2&\ge C_1\left(|s-t|^{2H}+|x-y|^{2H}\right)-\frac{(A_2(2H+1)T^{2H})^2}{4C_0}|t-s|^{2}\\
&\ge \frac{C_1}{2}\left(|s-t|^{2\tau_1}+|x-y|^{2\tau_2}\right),
\end{align*}
provided 
\begin{align*}
|s-t|\le \min\left\{\left(\frac{2C_0C_1}{A_1^2}\right)^{\frac{1}{H}},\left(\frac{2C_0C_1}{(A_2(2H+1)T^{2H})^2}\right)^{\frac{1}{2-2H}}\right\}=:c_8.
\end{align*}
In conclusion, if  $|s-t|^{2\tau_1}+|x-y|^{2\tau_2}\le C_0/C_2$ and $|t-s|\le c_8$ hold, then \eqref{gamtxsy} is valid for some constant $C>0$. Using equation (4.42)
in \cite{DKN07},
\begin{align*}
(\gamma_{t,x;s,y}^2-(\sigma_{t,x}-\sigma_{s,y})^2)((\sigma_{t,x}+\sigma_{s,y})^2-\gamma_{t,x;s,y}^2)=4(\sigma_{t,x}^2\sigma_{s,y}^2-\sigma_{t,x;s,y}^2),
\end{align*}
where $\sigma_{t,x;s,y}:=\operatorname{Cov}(V(t,x),V(s,y))$.
In order to extend \eqref{gamtxsy} to all $(s,y)$ and $(t,x)$ in $[t_0,T]\times[-M,M]$, it suffices to show that for any $(t,x)\neq(s,y)$,
\begin{align}\label{sig>0}
\sigma_{t,x}^2\sigma_{s,y}^2-\sigma_{t,x;s,y}^2> 0,
\end{align}
because the continuity of the function $(t,x,s,y)\mapsto\sigma_{t,x}^2\sigma_{s,y}^2-\sigma_{t,x;s,y}^2$ indicates that $\sigma_{t,x}^2\sigma_{s,y}^2-\sigma_{t,x;s,y}^2>c$ for some $c>0$ and all $(s,y)$ and $(t,x)$ in $[t_0,T]\times[-M,M]$ satisfying $|s-t|^{2\tau_1}+|x-y|^{2\tau_2}\ge C_0/C_2$ or $|t-s|\ge c_8$.
Observe that for any $\lambda\in\R$ and $s< t$,
\begin{align*}
\E[|V(t,x)-\lambda V(s,y)|^2]=&\int_s^t\int_{\R}|\F G_{t-r}(x-\cdot)(\xi)|^2\mu(d\xi)dr\\
&+\int_0^s\int_{\R}|\F \{G_{t-r}(x,\cdot)-\lambda G_{s-r}(y,\cdot)\}(\xi)|^2\mu(d\xi)dr.
\end{align*}
Therefore, $\E[|V(t,x)-\lambda V(s,y)|^2]=0$
if and only if $t=s$ and $x=y$,
then the proof of \eqref{sig>0} is finished by a similar argument of the proof of (4.44) in \cite{DKN07}. 
\qed

 \begin{tm}\label{0Id}
 Assume that $b\equiv 0$ and $\sigma=I_{d\times d}$ in Eq. \eqref{spde}.
 Let $I$ and $J$ be non-trivial compact sets in $(0,T]$ and $\R$, respectively.  If  $A \subseteq \mathbb{R}^{d}$  is a Borel set,  then there exists $C=C(A,I,J,H)>0$ such that
$$C^{-1}\operatorname{Cap}_{d-Q}(A) \le \mathbb{P}\left\{u(I\times J)\cap A \neq\emptyset\right\} \le C \mathscr{H}_{d-Q}(A),$$
where $Q=3/H$ for SHE and $Q=2/H$ for SWE.
 \end{tm}
Proof: By the assumption upon $I$ and $J$, there exists two rectangle{\color{black}s} $B_1$ and $B_2$ in $(0,T]\times\R$
such that $B_1\subset I\times J\subset B_2$. Noticing that 
$$\mathbb{P}\left\{u(B_1)\cap A \neq \emptyset\right\}\le\mathbb{P}\left\{u(I\times J)\cap A \neq \emptyset\right\}\le \mathbb{P}\left\{u(B_2)\cap A \neq \emptyset\right\},$$
then the desired result follows from Theorem \ref{Gausshit} and Lemmas \ref{Condi1} and \ref{Condi2}.
 \qed
 
Based on Theorem \ref{0Id}, we are in a position to prove Theorem \ref{0sigma}.

{\color{black}
 \textit{Proof of Theorem \ref{0sigma}:} 
%
 For an invertible matrix $\tilde\sigma\in\R^{d\times d}$ and a Borel set $B\subset \R^d$, we denote $$B_{\tilde \sigma}=\{\tilde\sigma^{-1} a,a\in B\}.$$
We also introduce $U:=\sigma^{-1}u$. Then
$U(t,x)=\int_0^t\int_{\R} G_{t-r}(x-z) W(dr,dz)$.
By Theorem \ref{0Id} and $\mathbb{P}\left\{u(I\times J)\cap A \neq \emptyset\right\}=\mathbb{P}\left\{U(I\times J)\cap A_{\sigma^{-1}} \neq \emptyset\right\}$,
we obtain 
\begin{equation}\label{Utxh}
C^{-1}\operatorname{Cap}_{d-Q}(A_{\sigma^{-1}})\le\mathbb{P}\left\{u(I\times J)\cap A \neq \emptyset\right\}\le C \mathscr{H}_{d-Q}(A_{\sigma^{-1}}).
\end{equation}
Notice that the Hausdorff measure and capacity of the image of a set under Lipschitz mappings is comparable with those of the original set (see e.g.   \cite[Theorem 2.8]{EG15} \& \cite[Theorem 5.2.1]{AH96}), i.e., for a Lipschitz function $f:\R^d\rightarrow \R^d$,
\begin{gather*}
\mathscr{H}_{d-Q}(f(A))\le C_1\mathscr{H}_{d-Q}(A),\\
\operatorname{Cap}_{d-Q}(f(A))\le C_2\operatorname{Cap}_{d-Q}(A),
\end{gather*}
for some constants $C_1,$ $C_2$ depending on $f$.
It follows from the invertibility of $\sigma$ that the mappings $\R^d\ni x\mapsto \sigma x$ and $\R^d\ni x\mapsto \sigma^{-1} x$ are Lipschitz, thus
$\mathscr{H}_{d-Q}(A_{\sigma^{-1}})\le C_3\mathscr{H}_{d-Q}(A)$
and
$\operatorname{Cap}_{d-Q}(A)\le C_4\operatorname{Cap}_{d-Q}(A_{\sigma^{-1}})$,
which together with \eqref{Utxh} completes the proof of Theorem \ref{0sigma}.
 \qed

We would like to mention that it is possible to extend the result about hitting probabilities in Theorem \ref{0sigma} for Eq. \eqref{spde} to the case of $b\neq constant$ and $\sigma=0$ by eliminating the drift term via Girsanov’s theorem (see e.g. \cite{DKN07}). By means of
studying the density functions of the associated solutions (see e.g. \cite{DKN09,DS10}), it is also 
worthwhile to further consider the hitting probabilities for Eq. (1.1) with $\sigma$ depending on $u$.

\section*{Acknowledgment}

This work is supported by National Natural Science Foundation of China, No. 11971470, No. 11926417, and No. 12101296, and Southern University of Science and Technology fund, No. Y01286232.

}

\bibliographystyle{amsalpha}
\bibliography{bib}

\newcommand{\etalchar}[1]{$^{#1}$}
\def\cprime{$'$}
\providecommand{\bysame}{\leavevmode\hbox to3em{\hrulefill}\thinspace}
\providecommand{\MR}{\relax\ifhmode\unskip\space\fi MR }
\providecommand{\MRhref}[2]{%
  \href{http://www.ams.org/mathscinet-getitem?mr=#1}{#2}
}
\providecommand{\href}[2]{#2}
\begin{thebibliography}{HHL{\etalchar{+}}17}

\bibitem[Adl90]{Adl90}
R.~Adler, \emph{An introduction to continuity, extrema, and related topics for
  general {G}aussian processes}, Institute of Mathematical Statistics Lecture
  Notes---Monograph Series, 12, Institute of Mathematical Statistics, Hayward,
  CA, 1990. \MR{1088478}

\bibitem[AH96]{AH96}
David~R. Adams and Lars~Inge Hedberg, \emph{Function spaces and potential
  theory}, Grundlehren der Mathematischen Wissenschaften [Fundamental
  Principles of Mathematical Sciences], vol. 314, Springer-Verlag, Berlin,
  1996. \MR{1411441}

\bibitem[AMN01]{AMN01}
E.~Al{\`o}s, O.~Mazet, and D.~Nualart, \emph{Stochastic calculus with respect
  to {G}aussian processes}, Ann. Probab. \textbf{29} (2001), no.~2, 766--801.
  \MR{1849177}

\bibitem[BJQS15]{BJQ15}
R.~Balan, M.~Jolis, and L.~Quer-Sardanyons, \emph{S{PDE}s with affine
  multiplicative fractional noise in space with index {$\frac14<H<\frac12$}},
  Electron. J. Probab. \textbf{20} (2015), no. 54, 36. \MR{3354614}

\bibitem[BJQS16]{BJQ16}
\bysame, \emph{S{PDE}s with rough noise in space: {H}\"older continuity of the
  solution}, Statist. Probab. Lett. \textbf{119} (2016), 310--316. \MR{3555303}

\bibitem[BLX09]{BLX09}
H.~Bierm{\'e}, C.~Lacaux, and Y.~Xiao, \emph{Hitting probabilities and the
  {H}ausdorff dimension of the inverse images of anisotropic {G}aussian random
  fields}, Bull. Lond. Math. Soc. \textbf{41} (2009), no.~2, 253--273.
  \MR{2496502}

\bibitem[CHL17]{CHL17(SINUM)}
Y.~Cao, J.~Hong, and Z.~Liu, \emph{Approximating stochastic evolution equations
  with additive white and rough noises}, SIAM J. Numer. Anal. \textbf{55}
  (2017), no.~4, 1958--1981. \MR{3686802}

\bibitem[CHL18]{CHL18(IMA)}
\bysame, \emph{Finite element approximations for second-order stochastic
  differential equation driven by fractional {B}rownian motion}, IMA J. Numer.
  Anal. \textbf{38} (2018), no.~1, 184--197. \MR{3800019}

\bibitem[Dal99]{Dal99}
R.~Dalang, \emph{Extending the martingale measure stochastic integral with
  applications to spatially homogeneous s.p.d.e.'s}, Electron. J. Probab.
  \textbf{4} (1999), no.\ 6, 29 pp.\ (electronic). \MR{1684157}

\bibitem[DF98]{DF98}
R.~Dalang and N.~Frangos, \emph{The stochastic wave equation in two spatial
  dimensions}, Ann. Probab. \textbf{26} (1998), no.~1, 187--212. \MR{1617046}

\bibitem[DKN07]{DKN07}
R.~Dalang, D.~Khoshnevisan, and E.~Nualart, \emph{Hitting probabilities for
  systems of non-linear stochastic heat equations with additive noise}, ALEA
  Lat. Am. J. Probab. Math. Stat. \textbf{3} (2007), 231--271. \MR{2365643}

\bibitem[DKN09]{DKN09}
\bysame, \emph{Hitting probabilities for systems for non-linear stochastic heat
  equations with multiplicative noise}, Probab. Theory Related Fields
  \textbf{144} (2009), no.~3-4, 371--427. \MR{2496438}

\bibitem[DKN13]{DKN13}
\bysame, \emph{Hitting probabilities for systems of non-linear stochastic heat
  equations in spatial dimension {$k\geq 1$}}, Stoch. Partial Differ. Equ.
  Anal. Comput. \textbf{1} (2013), no.~1, 94--151. \MR{3327503}

\bibitem[DSS05]{DS05}
R.~Dalang and M.~Sanz-Sol{\'e}, \emph{Regularity of the sample paths of a class
  of second-order spde's}, J. Funct. Anal. \textbf{227} (2005), no.~2,
  304--337. \MR{2168077}

\bibitem[DSS09]{DS09}
\bysame, \emph{H\"older-{S}obolev regularity of the solution to the stochastic
  wave equation in dimension three}, Mem. Amer. Math. Soc. \textbf{199} (2009),
  no.~931, vi+70. \MR{2512755}

\bibitem[DSS10]{DS10}
\bysame, \emph{Criteria for hitting probabilities with applications to systems
  of stochastic wave equations}, Bernoulli \textbf{16} (2010), no.~4,
  1343--1368. \MR{2759182}

\bibitem[DSS15]{DS15}
\bysame, \emph{Hitting probabilities for nonlinear systems of stochastic
  waves}, Mem. Amer. Math. Soc. \textbf{237} (2015), no.~1120, v+75.
  \MR{3401290}

\bibitem[EG15]{EG15}
Lawrence~C. Evans and Ronald~F. Gariepy, \emph{Measure theory and fine
  properties of functions}, revised ed., Textbooks in Mathematics, CRC Press,
  Boca Raton, FL, 2015. \MR{3409135}

\bibitem[HHL{\etalchar{+}}17]{HHLNT17(AOP)}
Y.~Hu, J.~Huang, K.~L\^{e}, D.~Nualart, and S.~Tindel, \emph{Stochastic heat
  equation with rough dependence in space}, Ann. Probab. \textbf{45} (2017),
  no.~6B, 4561--4616. \MR{3737918}

\bibitem[HHN14]{HHN14}
Y.~Hu, J.~Huang, and D.~Nualart, \emph{On {H}\"older continuity of the solution
  of stochastic wave equations in dimension three}, Stoch. Partial Differ. Equ.
  Anal. Comput. \textbf{2} (2014), no.~3, 353--407. \MR{3255232}

\bibitem[HLN12]{HLN12}
Y.~Hu, F.~Lu, and D.~Nualart, \emph{Feynman-{K}ac formula for the heat equation
  driven by fractional noise with {H}urst parameter {$H<1/2$}}, Ann. Probab.
  \textbf{40} (2012), no.~3, 1041--1068. \MR{2962086}

\bibitem[HNS11]{HNS11}
Y.~Hu, D.~Nualart, and J.~Song, \emph{Feynman-{K}ac formula for heat equation
  driven by fractional white noise}, Ann. Probab. \textbf{39} (2011), no.~1,
  291--326. \MR{2778803}

\bibitem[Jol10]{Jol10}
M.~Jolis, \emph{The {W}iener integral with respect to second order processes
  with stationary increments}, J. Math. Anal. Appl. \textbf{366} (2010), no.~2,
  607--620. \MR{2600506}

\bibitem[MT02]{MT02}
C.~Mueller and R.~Tribe, \emph{Hitting properties of a random string},
  Electron. J. Probab. \textbf{7} (2002), no. 10, 29 pp. (electronic).
  \MR{1902843}

\bibitem[NV09]{NV09}
E.~Nualart and F.~Viens, \emph{The fractional stochastic heat equation on the
  circle: time regularity and potential theory}, Stochastic Process. Appl.
  \textbf{119} (2009), no.~5, 1505--1540. \MR{2513117}

\bibitem[SSS02]{SS02}
M.~Sanz-Sol{\'e} and M.~Sarr{\`a}, \emph{H\"older continuity for the stochastic
  heat equation with spatially correlated noise}, Seminar on {S}tochastic
  {A}nalysis, {R}andom {F}ields and {A}pplications, {III} ({A}scona, 1999),
  Progr. Probab., vol.~52, Birkh\"auser, Basel, 2002, pp.~259--268.
  \MR{1958822}

\bibitem[Wal86]{Wal86}
J.~Walsh, \emph{An introduction to stochastic partial differential equations},
  \'{E}cole d'\'et\'e de probabilit\'es de {S}aint-{F}lour, {XIV}---1984,
  Lecture Notes in Math., vol. 1180, Springer, Berlin, 1986, pp.~265--439.
  \MR{876085}

\end{thebibliography}

\end{document}